\documentclass[11pt]{article}
\usepackage{mathrsfs,amsmath}
\usepackage[thmmarks,amsmath,amsthm]{ntheorem}

\usepackage{amsfonts, psfrag, graphicx, amssymb,enumitem,indentfirst,float,geometry,color,enumerate,graphicx,subfigure,algorithm,algpseudocode,pifont,hyperref}

\allowdisplaybreaks[4]
\numberwithin{equation}{section}
\numberwithin{figure}{section}

\usepackage{lipsum}

\newtheorem{thm}{Theorem}[section]
\newtheorem{lemma}{Lemma}[section]

\newcommand{\commentout}[1]{{}} 

\newcommand{\abs}[1]{\left|#1\right|}
\newcommand{\norm}[1]{\left\|#1\right\|}

\newcommand{\bfb}{{\bf b}}

\newcommand{\bfc}{{\bf c}}

\newcommand{\bfn}{{\bf n}}

\newcommand{\bfdelta}{\boldsymbol{\delta}}

\newcommand{\bfgamma}{\boldsymbol{\gamma}}

\newcommand{\bfLambda}{\boldsymbol{\Lambda}}

\newcommand{\bfpsi}{\boldsymbol{\psi}}

\begin{document}
\title{Nonconforming Immersed Finite Element Spaces \\
For Elliptic Interface Problems}
\author{Ruchi Guo\thanks{Department of Mathematics, Virginia Tech, Blacksburg, VA 24061 (ruchi91@vt.edu) } 
\and Tao Lin\thanks{Department of Mathematics, Virginia Tech, Blacksburg, VA 24061 (tlin@math.vt.edu)} 
\and Xu Zhang\thanks{Department of Mathematics and Statistics, Mississippi State University, Mississippi State, MS 39762 (xuzhang@math.msstate.edu)}}
\date{}
\maketitle

\begin{abstract}
In this paper, we use a unified framework introduced in \cite{2016GuoLin} to study two classes of nonconforming immersed finite element (IFE) spaces with integral value degrees of freedom. The shape functions on interface elements are piecewise polynomials defined on sub-elements separated either by the actual interface or its line approximation. In this unified framework, we use the invertibility of the well known Sherman-Morison systems to prove the existence and uniqueness of shape functions on each interface element in either rectangular or triangular mesh. Furthermore, we develop a multi-edge expansion for piecewise functions and a group of identities for nonconforming IFE functions which enable us to show that these IFE spaces have the optimal approximation capability.
\end{abstract}

\section{Introduction}

Consider the classical second order elliptic interface problem:
\begin{align}
\label{eq1_1}
 -\nabla\cdot(\beta\nabla u)=f, \;\;\;\; & \text{in} \; \Omega^-  \cup \Omega^+, \\
 u=g, \;\;\;\; &\text{on} \; \partial\Omega,
\end{align}
where the domain $\Omega\subseteq \mathbb{R}^2$ is assumed to be separated by an interface curve $\Gamma$ into two subdomains $\Omega^+$ and $\Omega^-$. The diffusion coefficient $\beta(X)$ is a piecewise constant:
\begin{equation*}
\beta(X)=
\left\{\begin{array}{cc}
\beta^- & \text{if} \; X\in \Omega^- ,\\
\beta^+ & \text{if} \; X\in \Omega^+,
\end{array}\right.
\end{equation*}
and the exact solution $u$ is required to satisfy the jump conditions:
\begin{eqnarray}
[u]_{\Gamma} &=& 0, \label{eq1_3} \\
\big[\beta \nabla u\cdot \mathbf{n}\big]_{\Gamma} &=& 0, \label{eq1_4}
\end{eqnarray}
where $\mathbf{ n}$ is the unit normal vector to the interface $\Gamma$. Here and from now on, for every piecewise function $v$ defined as
\begin{equation*}
v=
\left\{\begin{array}{cc}
v^-(X) & \text{if} \; X\in \Omega^- , \\
v^+(X) & \text{if} \; X\in \Omega^+ ,
\end{array}\right.
\end{equation*}
we adopt the notation $[v]|_{\Gamma}=v^+|_{\Gamma}-v^-|_{\Gamma}$.

The IFE method was first introduced in \cite{1998Li} for solving an 1D elliptic interface problem with meshes independent of the interface. Extensions to 2D elliptic interface problems include IFE functions defined by conforming $P_1$ polynomials \cite{2007GongLiLi,2005HouLiu,2004LiLinLinRogers,2003LiLinWu}, conforming $Q_1$ polynomials \cite{2009HeTHESIS,2008HeLinLin, 2001LinLinRogersRyan,2012LinZhang}, nonconforming $P_1$ polynomials (Crouzeix-Raviart) \cite{2010KwakWeeChang}, and nonconforming rotated-$Q_1$ polynomials \cite{2013LinSheenZhang,2015LinSheenZhang,2013ZhangTHESIS}. IFE functions in these articles are $H^1$ functions defined with a line approximating the original interface curve
in each interface element. Recently, the authors in \cite{2016GuoLin,2015GuzmanSanchezSarkis} developed IFE spaces according to the original interface curve where the local degrees of freedom are of Lagrange type. The goal of the present article is to develop and analyze IFE spaces constructed according to the actual interface curve and the integral value degrees of freedom on element edges.

Two main features motivate us to consider IFE functions with integral value degrees of freedom and with the actual interface curve instead of its line approximation. First, as observed in \cite{2015LinSheenZhang,2013ZhangTHESIS}, IFE functions of this type have less severe discontinuity across interface edges because their continuity across an element edge is weakly enforced over the entire edge in an average sense. Compared to IFE spaces with Lagrange type of degrees of freedom \cite{2015LinLinZhang}, an IFE method based on integral value degrees of freedom does not necessarily require any penalty terms imposed over interface edges to ensure the derivation of optimal error bounds and better numerical performance. The second motivation is our desire to develop higher degree IFE spaces for which using a line to approximate the interface curve is
not sufficient anymore because of the $O(h^2)$ accuracy limitation for the line to approximate a curve. Even though the IFE functions in this article are still lower degree $P_1$ or rotated-$Q_1$
polynomials, we hope their investigation can become a precursor to the development of higher degree IFE spaces. In addition, we will demonstrate later that the framework presented here can also be applied to IFE spaces with rotated-$Q_1$ elements \cite{2015LinSheenZhang, 2013ZhangTHESIS} and Crouzeix-Raviart elements \cite{2010KwakWeeChang}, respectively, where IFE functions are defined according to a line approximating the interface curve in each interface element.

Even though the new IFE spaces presented here seem to be natural because they are constructed locally on each interface element according to the actual interface curve of the problem to be solved,
the related investigation faces a few hurdles. The first one is that the new IFE functions are discontinuous in each interface element except for trivial interface geometry because,
in general, two distinct polynomials cannot perfectly match each other on a curve. In contrast, almost all IFE spaces in the literature are continuous in each element. This lack of continuity leads to a lower regularity
of IFE functions in interface elements such that related error analysis demands new approaches different from those in the literature
\cite{2008HeLinLin, 2014JiChenLi, 2004LiLinLinRogers,1982Xu, 2013ZhangTHESIS}. Another issue is that the interpolation error analysis technique based on
the multi-point Taylor expansion in the literature is not applicable here because new IFE functions are constructed with integral value degrees of freedom instead of the Lagrange type degrees of freedom.

This article is organized as follows. In Section 2, we introduce some basic notations, assumptions and known results to be used in this article. In Section 3, we extend the multi-point Taylor expansion established in \cite{2016GuoLin,2009HeTHESIS,2004LiLinLinRogers,2013ZhangTHESIS} to a multi-edge expansion for piecewise $C^2$ functions such that the new expansion can handle integral value degrees of freedom. Estimates for remainders in this new expansion are also developed in this section. In Section 4, we show that the integral value degrees of freedom imposed on each edge and the approximated jump conditions together yield a Sherman-Morrison system for determining coefficients in IFE shape functions
on interface elements. We show that the unisolvence and boundedness of IFE spaces follows from the well-known invertibility of the Sherman-Morrison system.
A group of fundamental identities such as partition of unity are also derived for new IFE shape functions. In Section 5, we establish the optimal approximation capability for IFE spaces with the integral degrees of freedom defined either according to the actual interface or to a line approximating the interface curve \cite{2010KwakWeeChang,2013ZhangTHESIS}. Finally, in Section 6, we present some numerical examples.

\section{Preliminaries}
\label{sec:preliminaries}
Throughout this article, we adopt the notations used in \cite{2016GuoLin}, and we recall some of them for reader's convenience. We assume that $\Omega\subset\mathbb{R}^2$ is a bounded domain that is a union of finitely many rectangles, and that $\Omega$ is separated by an interface curve $\Gamma$  into two subdomains $\Omega^+$ and $\Omega^-$ such that $\overline{\Omega} = \overline{\Omega^+} \cup \overline{\Omega^-} \cup \Gamma$. For any measurable subset $\tilde \Omega \subseteq \Omega$, we define the standard Sobolev spaces $W^{k,p}(\tilde \Omega)$ and the associated norm, $\|\cdot\|_{k,p,\tilde \Omega}$ and semi-norm, $|v|_{k,p,\tilde \Omega}=\|D^{\alpha}v\|_{0,p,\tilde \Omega}$, for $|\alpha|=k$. The corresponding Hilbert space is $H^k(\tilde \Omega)=W^{k,2}(\tilde \Omega)$. When
$\tilde \Omega^s = \tilde \Omega \cap \Omega^s \not = \emptyset, s = \pm$, we let
\begin{equation*}
PH^k_{int}(\tilde \Omega)=\{ u: u|_{\tilde{\Omega}^s}\in H^k(\tilde \Omega^s),\; s=\pm; \; [u]=0 \; \text{and} \; [\beta\nabla u\cdot \mathbf{n}_\Gamma]=0\; \text{on}\; \Gamma \cap \tilde \Omega \},
\end{equation*}
\begin{equation*}
PC^k_{int}(\tilde \Omega)=\{ u: u|_{\tilde \Omega^s}\in C^k(\tilde \Omega^s),\; s=\pm; \; [u]=0 \; \text{and} \; [\beta\nabla u\cdot \mathbf{n}_{\Gamma}]=0\; \text{on}\; \Gamma \cap \tilde \Omega \},
\end{equation*}
with the associated norms:
\begin{equation*}
\|\cdot\|^2_{k,\tilde \Omega}=\|\cdot\|^2_{k,\tilde \Omega^+}+\|\cdot\|^2_{k,\tilde \Omega^-}, \;\;\;\;\; |\cdot|^2_{k,\tilde \Omega}=|\cdot|^2_{k,\tilde \Omega^+}+|\cdot|^2_{k,\tilde \Omega^-},
\end{equation*}
\begin{equation*}
\|\cdot\|_{k,\infty,\tilde \Omega}=\max(\|\cdot\|_{k,\infty,\tilde \Omega^+} \;,\; \|\cdot\|_{k,\infty,\tilde \Omega^-}), \;\;\;\;\; |\cdot|_{k,\infty,\tilde \Omega}=\max(|\cdot|_{k,\infty,\tilde \Omega^+} \;,\; |\cdot|_{k,\infty,\tilde \Omega^-}).
\end{equation*}

Let $\mathcal{T}_h$ be a Cartesian triangular or rectangular mesh of the domain $\Omega$ with the maximum length of edge $h$. An element $T\in \mathcal{T}_h$ is called an interface element provided the interior of $T$ intersects with the interface $\Gamma$; otherwise, we name it a non-interface element. We let $\mathcal{T}^i_h$ and $\mathcal{T}^n_h$ be the set of interface elements and non-interface elements, respectively. Similarly, $\mathcal{E}^i_h$ and $\mathcal{E}^n_h$ are sets of interface edges and non-interface edges, respectively. Besides, we assume that $\mathcal{T}_h$ satisfies the following hypotheses \cite{2009HeLinLin}, when the mesh size $h$ is small enough:

\begin{itemize}[leftmargin=30pt]
  \item [\textbf{(H1)}] The interface $\Gamma$ cannot intersect an edge of any element at more than two points unless the edge is part of $\Gamma$. \label{assump_h1}
  \item [\textbf{(H2)}] If $\Gamma$ intersects the boundary of an element at two points, these intersection points must be on different edges of this element. \label{assump_h2}
  \item [\textbf{(H3)}] The interface $\Gamma$ is a piecewise $C^2$ function, and the mesh $\mathcal{T}_h$ is formed such that the subset of $\Gamma$ in every interface element $T\in\mathcal{T}^i_h$ is $C^2$- continuous. \label{assump_h3}
  \item[\textbf{(H4)}] The interface $\Gamma$ is smooth enough so that $PC^2_{int}(T)$ is dense in $PH^2_{int}(T)$ for every interface element $T\in\mathcal{T}^i_h$. \label{assump_h4}
\end{itemize}

On an element $T \in \mathcal{T}_h$, we consider the local finite element space $(T, \Pi_T, \Sigma_T)$ with
\begin{eqnarray}
\Pi_T &=& \begin{cases}
\textrm{Span} \{ 1,x,y\}, & \text{for Crouzeix-Raviart (C-R) finite element functions}, \\
\textrm{Span} \{ 1,x,y,x^2-y^2 \}, & \text{for rotated-$Q_1$ finite element functions},
\end{cases} \label{fe_cases}\\
\Sigma_T &=& \left\{ \frac{1}{|b_i|}\int_{b_i}\psi_T(X)ds: i \in \mathcal{I}, \; \forall \psi_T \in \Pi_T \right\}, \label{integral_DOF}
\end{eqnarray}
where $b_i, i\in \mathcal{I}$ are edges of the element $T$, $\mathcal{I} = \{1, 2, \cdots, DOF(T)\}$, $DOF(T) = 3$ or $4$ depending on whether $T$ is triangular or rectangular. 
In addition, let $M_i$ be the midpoint of the edge $b_i$, $i\in\mathcal{I}$. Recall from \cite{1992RanacherTurek}
that 
$(T, \Pi_T, \Sigma_T)$ has a set of shape functions $\psi_{i,T}, i\in \mathcal{I}$ such that
\begin{equation}
\label{eqn_psi_iT_values_bounds}
\frac{1}{|b_j|}\int_{|b_j|}\psi_{i,T}(X)ds=\delta_{ij},~~\norm{\psi_{i,T}}_{\infty, T} \leq C,~~\norm{\nabla \psi_{i,T}}_{\infty, T} \leq Ch^{-1}, \;\; i,j\in \mathcal{I},
\end{equation}
where $\delta_{ij}$ is the \textit{Kronecker} delta function.

Furthermore, we let $\rho=\beta^-/\beta^+$, and, on any $T\in \mathcal{T}^i_h$, we use $D$, $E$ to denote the intersection points of $\Gamma$ and $\partial T$ and let $l$ be the line connecting $D$ and $E$. Let $\bar{\mathbf{ n}}=(\bar{n}_x,\bar{n}_y)^t$ and $\mathbf{ n}(\widetilde{X})=(\tilde{n}_x(\widetilde{X}),\tilde{n}_y(\widetilde{X}))^t$ be the normal vector to $l$ and to $\Gamma$ at $\widetilde{X}\in\Gamma$, respectively. In the following discussion, $s$ is the index that is either - or +, and $s'$ takes the opposite sign whenever a formula have them both. And let $F$ be an arbitrary point either on the line $l$ or the interface curve $\Gamma\cap T$. We associate the point $F$ with a vector $\mathbf{v}(F)=(v_x(F),v_y(F))^t$ such that the following two cases will be considered:
\begin{itemize}
\label{partition_case1}\item[1] If $F\in\Gamma\cap T$ but $F\neq D,E$, then $\mathbf{ v}(F)=\mathbf{ n}(F)$ and $T$ is partitioned by $\Gamma$ into two subelements $T^s_{curve}=T^s$, $s=\pm$.

\label{partition_case2}\item[2] If $F\in l$, then let $\mathbf{ v}(F)=\bar{\mathbf{ n}}$ and $T$ is partitioned by $l$ into two subelements $T^s_{line}$, $s=\pm$.
\end{itemize}

Lemma 3.1 and Remark 3.1 in \cite{2016GuoLin} provide one critical ingredient in our analysis: on a mesh fine enough, there exists a constant $C$ such that
\begin{equation}
\label{estimate_nn}
\mathbf{ v}(F)\cdot\bar{\mathbf{ n}}\geqslant1-Ch^2.
\end{equation}
As in \cite{2016GuoLin}, we will employ the following matrices:
\begin{equation}
\label{mat_curve}
M^s(\widetilde{X})=
\left(\begin{array}{cc}
\tilde{n}^2_y(\widetilde{X})+\beta^s/\beta^{s'} \tilde{n}^2_x(\widetilde{X}) & (\beta^s/\beta^{s'}-1)\tilde{n}_x(\widetilde{X})\tilde{n}_y(\widetilde{X}) \\
(\beta^s/\beta^{s'}-1)\tilde{n}_x(\widetilde{X})\tilde{n}_y(\widetilde{X}) & \tilde{n}^2_x(\widetilde{X})+\beta^s/\beta^{s'} \tilde{n}^2_y(\widetilde{X})
\end{array}\right),
\end{equation}
\begin{equation}
\begin{split}
\label{mat_line}
\overline{M}^s(F)
&=\frac{1}{\bar{\mathbf{ n}}\cdot\mathbf{ n}(F)}
\left(\begin{array}{cc}
\bar{n}_yn_y(F)+\beta^s/\beta^{s'}\bar{n}_xv_x(F) & -\bar{n}_xv_y(F)+\beta^s/\beta^{s'}\bar{n}_xv_y(F) \\
-\bar{n}_yn_x(F)+\beta^s/\beta^{s'}\bar{n}_yv_x(F) & \bar{n}_xv_x(F)+\beta^s/\beta^{s'}\bar{n}_yv_y(F)
\end{array}\right),
\end{split}
\end{equation}
where $s=\pm$ and $\overline{M}^s(F)$ is well defined since \eqref{estimate_nn} implies that $\bar{\mathbf{ n}}\cdot\mathbf{ n}(F)>0$ when $h$ is small enough.

\section{Multi-edge Taylor Expansions on Interface Elements}
In this section, we derive a multi-edge expansion for a function $u$ on an interface element to handle integral degrees of freedom. We will show that the integral value $\frac{1}{|b_i|}\int_{b_i}u(X)ds$, $i \in \mathcal{I}$ can be expressed in terms of $u$ and its derivatives for various configurations of the interface and edges. Estimates for the remainders of this expansion will be given.


We partition the index set $\mathcal{I}$ into three subsets $\mathcal{I}^-=\{ i: b_i\subseteq \overline{T^-} \}$, $\mathcal{I}^+=\{ i:  b_i\subseteq \overline{T^+} \}$ and $\mathcal{I}^{int}=\{i: b_i\cap T^s \neq \emptyset, s=\pm \}$. Given an edge $b_i$, for every point $P\in b_i$ and $X\in T$, we note that $Y_i(t,P,X)=tP+(1-t)X, t \in [0, 1]$ is a point on the line segment connecting $P$ and $X$. We note that for some points $X$ and $P$, the line $\overline{PX}$ may intersect the curve $\Gamma\cap T$ at multiple points. Define
$$
T_{int}=\{ X\in T: \textrm{there exists a point }Y \in T\cap \Gamma \textrm{, such that } \overline{XY} \textrm{ is a tangent line to }\Gamma\textrm{ at } Y \}
$$
which is actually formed by the line segments inside $T$ each of which is tangent to $T\cap\Gamma$ at an end point $Y$. Lemma 3.1 in \cite{2016GuoLin} shows that  $|T_{int}|<Ch^3$ for a mesh fine enough.

First we derive the multi-edge expansion for  a point $X\in T_{non} = T\setminus T_{int}$. For convenience, we define $T^s_{non}=T_{non}\cap T^s$. We note that for
any $P\in\partial T$, the line segment $\overline{PX}$ intersects with $\Gamma\cap T$ either at no point or at just one point. In the second case, $X$ and $P$ sit on different sides
of $\Gamma\cap T$ and we denote the intersection point by $\widetilde{Y}_i=Y_i(\tilde{t},P,X)$ for a $\tilde{t}_i=\tilde{t}_i(P,X)\in [0,1]$.
Consider a piecewisely defined function $R_i \,:\, b_i\times T_{non} \rightarrow \mathbb{R}$ given by
\begin{equation}
\label{reminder_1}
R_i(P,X) = \left\{
\begin{aligned}
&\int_0^1(1-t)\frac{d^2}{dt^2}u^s(Y_i(t,P,X))dt, & \text{if} ~ P\in T^s\cap b_i, ~X \in T_{non}^s, \\
&R_{i1}(P,X)+R_{i2}(P,X)+R_{i3}(P,X), & \text{if} ~ P\in T^{s'}\cap b_i, ~X \in T_{non}^s,
\end{aligned}
\right.
\end{equation}
where
\begin{equation}
\begin{cases}
\label{reminder_2}
&R_{i1}(P,X)=\int_{0}^{\tilde{t}_i}(1-t)\frac{d^2u^s}{dt^2}(Y_i(t,P,X))dt, \\
&R_{i2}(P,X)=\int_{\tilde{t}_i}^{1}(1-t)\frac{d^2u^{s'}}{dt^2}(Y_i(t,P,X))dt,\\
&R_{i3}(P,X)=(1-\tilde{t}_i)\int_{0}^{\tilde{t}_i}\frac{d}{dt} \left( (M^s(\widetilde{Y}_i)-I)\nabla u^s(Y_i(t,X))\cdot (P-X) \right)dt.
\end{cases}
\end{equation}

For $u\in PC^2_{int}(T)$, recall the multi-point Taylor expansion formulation and the estimates of \eqref{reminder_1} and \eqref{reminder_2} in \cite{2016GuoLin}:
\begin{equation}
u^s(P)=u^s(X)+ \nabla u^s(X)\cdot(P-X) + R_i(P,X), ~ \text{if} \; P\in T^s\cap b_i, \; X \in T^s_{non},
 \label{eq_expan_same_side}
\end{equation}
\begin{equation}
\begin{split}
\label{eq_expan_diff_side}
u^{s'}(P)=&u^{s}(X)+\nabla u^{s}(X)\cdot (P-X)+ \left(\left( M^{s}(\widetilde{Y}_i)-I \right)\nabla u^{s}(X)\right) \cdot (P-\widetilde{Y}_i)\\
& + R_i(P,X), ~\text{if} \; P\in T^{s'}\cap b_i, \; X \in T^{s}_{non},
\end{split}
\end{equation}
and for any fixed $P\in b_i$,
\begin{equation}
\label{reminder_estimates}
\norm{R_i(P,\cdot)}_{0,T_{non}^s}\leqslant Ch^2|u|_{2,T}, ~ s=\pm, ~\forall i\in\mathcal{I}.
\end{equation}

Integrating \eqref{eq_expan_same_side} and \eqref{eq_expan_diff_side} on each edge $b_i$ with respect to $P$, we obtain the following
multi-edge expansion for $u^s(X)$ with $X \in T_{non}^s$:
\begin{equation}
\label{eq_integral_expan_1}
\frac{1}{|b_i|}\int_{b_i}u^s(P)ds(P)=u^s(X)+ \nabla u^s(X)\cdot(M_i-X) + \mathcal{R}_i(X), ~ i\in \mathcal{I}^s,
\end{equation}
\begin{equation}
\label{eq_integral_expan_2}
\begin{split}
\frac{1}{|b_i|}\int_{b_i}u^{s'}(P)ds(P)=&u^{s}(X)+\nabla u^{s}(X)\cdot (M_i-X) + \mathcal{R}_i(X) \\
&+\frac{1}{|b_i|}\int_{b_i}  \left(\left( M^{s}(\widetilde{Y}_i)-I \right)\nabla u^{s}(X)\right) \cdot (P-\widetilde{Y}_i) ds(P), ~ i\in \mathcal{I}^{s'},
\end{split}
\end{equation}
\begin{equation}
\label{eq_integral_expan_3}
\begin{split}
\frac{1}{|b_i|}\int_{b_i}u(P)ds(P)=&u^{s}(X)+\nabla u^{s}(X)\cdot (M_i-X) + \mathcal{R}_i(X) \\
&+\frac{1}{|b_i|}\int_{b_i\cap T^{s'}}  \left(\left( M^{s}(\widetilde{Y}_i)-I \right)\nabla u^{s}(X)\right) \cdot (P-\widetilde{Y}_i) ds(P), ~ i\in \mathcal{I}^{int},
\end{split}
\end{equation}
where
\begin{equation}
\label{reminder_integral}
\begin{aligned}
& \mathcal{R}_i(X)=\frac{1}{|b_i|}\int_{b_i}R_i(P,X)ds(P).
\end{aligned}
\end{equation}

Now we estimate these reminders $\mathcal{R}_i$.
\begin{lemma}
\label{reminder_estimate}
Assume $u\in PC^2_{int}(T)$, then there exists a constant $C$ independent of the interface location, such that
\begin{equation}
\label{integral_reminder_estimates}
\norm{\mathcal{R}_i}_{0,T_{non}^s}\leqslant Ch^2 |u|_{2,T}, ~ s=\pm,~ \forall i\in\mathcal{I}.
\end{equation}
\end{lemma}
\begin{proof}
By the estimates \eqref{reminder_estimates} and Minkowski inequality, we have
\begin{equation*}
\begin{split}
\norm{\mathcal{R}_i}_{0,T_{non}^s}=&\left( \int_{T_{non}^s} \left( \frac{1}{|b_i|}\int_{b_i}R_i(P,X)ds(P) \right)^2 dX \right)^{\frac{1}{2}} \\
\leqslant & \frac{1}{|b_i|} \int_{b_i} \left( \int_{T_{non}^s} (R_i(P,X))^2 dX \right)^{\frac{1}{2}} ds(P) \\
\leqslant & \frac{Ch^2}{|b_i|} \int_{b_i} |u|_{2,T} \; ds(P) \leqslant Ch^2|u|_{2,T}.
\end{split}
\end{equation*}
~~~\end{proof}

We now consider the multi-edge expansion for $X\in T_{int}$. We start from the  following first order multi-point Taylor expansion:
\begin{equation}
\begin{split}
\label{1st_expan_1}
u^s(P)=&u^s(X)+R_i(P,X),~~~~~ P\in b_i, ~X\in T_{int}
\end{split}
\end{equation}
where $R_i:b_i\times T_{int}\to\mathbb{R}$ is a function defined by
\begin{equation}
\begin{split}
\label{1st_expan_2}
R_i(P,X)=\int^1_0\frac{d u}{dt}(Y_i(t,P,X))dt=\int^1_0 \nabla u(Y_i(t,P,X))\cdot(P-X) dt
\end{split}
\end{equation}
Integrating \eqref{1st_expan_2} on each edge $b_i$ with respect to $P$, we obtain the following multi-edge expansion:
\begin{equation}
\label{1st_multiedge_1}
\frac{1}{|b_i|}\int_{b_i}u(P)ds(P)=u(X)+\mathcal{R}_i(X), ~~ i\in\mathcal{I},~~~~~ \textrm{where }~~ \mathcal{R}_i(X)=\frac{1}{|b_i|}\int_{b_i}R_i(P,X)ds(P).
\end{equation}

The following lemma will be used to estimate the reminders in this expansion.
\begin{lemma}
\label{lem_estim_Tint1}
There exists a constant $C$ independent of the interface location such that for fixed $t$
\begin{equation}
\label{lem_estim_Tint1_eq1}
\| R_i(P,\cdot) \|_{0,T_{int}}\leqslant Ch^2 \| u \|_{1,6,T}
\end{equation}
\end{lemma}
\begin{proof}
We consider the linear mapping $\xi: X \rightarrow \hat{X}=tP+(1-t)X$ for $t \in [0, 1]$ and $P\in b_i$ which maps
$T_{int}$ to
$$
\hat{T}_{int}(t)=\{ tP+(1-t)X : X\in T_{int} \}.
$$
Since $\xi$ is a linear mapping, $|\hat{T}_{int}(t)|=(1-t)^2|T_{int}|\leqslant C(1-t)^2h^3$. Now by H\"older's inequality and following the similar idea in \cite{1998ChenZou}, we have
\begin{equation*}
\begin{split}
\left( \int_{T_{int}} \left( \nabla u(Y_i)\cdot (P-X) \right)^2 dX \right)^{1/2} & \leqslant h \left( \int_{T_{int}} \abs{\nabla u(Y_i(t,P,X))}^2 dX \right)^{1/2} \\
&=(1-t)^{-1}h \left( \int_{\hat{T}_{int}} \abs{\nabla u(\hat{X})}^2 d\hat{X} \right)^{1/2} \\
&\leqslant (1-t)^{-1}h \left( \int_{\hat{T}_{int}}1^{3/2} d\hat{X} \right)^{1/3} \left( \int_{\hat{T}_{int}}|\nabla u(\hat{X})|^6 d\hat{X} \right)^{1/6}\\
&\leqslant (1-t)^{-1/3}h^2 \| u \|_{1,6,T}.
\end{split}
\end{equation*}
Then by Minkovski's inequality and the estimate above, we have
\begin{equation*}
\begin{split}
\| R_i(P,\cdot) \|_{0,T_{int}}= &\left(  \int_{T_{int}} \left( \int^1_0 \nabla u(Y_i)\cdot(P-X) dt \right)^2   dX  \right)^{1/2}\\
\leqslant&  \int^1_{0} \left( \int_{T_{int}} \left( \nabla u(Y_i)\cdot(P-X)\right)^2 dX \right)^{1/2} dt\\
\leqslant& h^2 \|u\|_{1,6,T} \int^1_0 (1-t)^{-1/3} dt =\frac{3}{2}h^2 \|u\|_{1,6,T}
\end{split}
\end{equation*}
which completes the proof.
\end{proof}

Finally, we give estimates for the remainder in the multi-edge expansion \eqref{1st_multiedge_1} in the following lemma.
\begin{lemma}
There exists a constant $C$ independent of the interface location such that 
\begin{equation}
\label{1st_multiedge_est}
\| \mathcal{R}_i \|_{0,T_{int}}\leqslant Ch^2 \|u\|_{1,6,T},~i \in \mathcal{I}.
\end{equation}
\end{lemma}
\begin{proof}
The results follow from Lemma \ref{lem_estim_Tint1} and arguments similar to those used in the proof for Lemma \ref{reminder_estimate}.
\end{proof}

\section{IFE Spaces and Their Properties}
\label{sec:IFE Spaces and Their Properties}

In this section, we use the finite element space $(T,\Pi_T,\Sigma_T)$ for $T \in \mathcal{T}_h$ to develop the nonconforming $P_1$ and rotated $Q_1$ IFE spaces with integral degrees of freedom. First we prove the unisolvence of the immersed finite elements on interface elements by the invertibility of the Sherman-Morison system. Then we present a few properties of IFE spaces which play important roles in the analysis of approximation capabilities. We note the framework presented here provides a unified approach for both nonconforming $P_1$ and rotated $Q_1$ IFE spaces developed in the literature \cite{2010KwakWeeChang,2015LinSheenZhang,2013ZhangTHESIS} and the new ones defined with the actual interface curve.

\subsection{Construction of IFE Spaces}

First, on every element $T\in \mathcal{T}_h$, we have the standard local finite element space
\begin{equation}
\label{stand_fe_loc}
S_h^{non}(T)= \textrm{Span} \{ \psi_{i,T}:\; i\in\mathcal{I} \},
\end{equation}
where $\psi_{i,T}, i\in\mathcal{I}$ are the shape functions satisfying \eqref{eqn_psi_iT_values_bounds}. Naturally \eqref{stand_fe_loc} can be used as the local IFE space on each non-interface element $T \in \mathcal{T}_h^n$. So we focus on constructing the local IFE space on every interface element.
\commentout{
Specifically, we will show that the IFE shape functions exist and can be uniquely determined by the integral degrees of freedom $\Sigma_T$ under some suitable interface conditions. And these shape functions then can span the local IFE space on each interface element.
}

The main task is to construct IFE shape functions on interface elements. We consider the IFE functions in the following form of piecewise polynomials
\begin{equation}
\label{ife_piecewise}
\phi_T(X) =
\left\{
\begin{aligned}
\phi^{-}_T(X)\in \Pi_T \;\;\;\;\; & \text{if} \;\; X\in T_{p}^-, \\
\phi^{+}_T(X)\in \Pi_T \;\;\;\;\; & \text{if} \;\; X\in T_{p}^+,
\end{aligned}
\right.
\end{equation}
where $p=curve$ or $line$, as described in Section \ref{sec:preliminaries}, such that the jump conditions \eqref{eq1_3} and \eqref{eq1_4} are satisfied in the following approximate sense:
\begin{eqnarray}
&&\begin{cases}
\phi^{-}_T|_{l}=\phi^{+}_T|_{l}, &\text{($T$ is a triangular element)}, \\
\phi^{-}_T|_{l}=\phi^{+}_T|_{l}, ~~\partial_{xx}(\phi^{-}_T)=\partial_{xx}(\phi^{+}_T), &\text{($T$ is a rectangular element)},
\end{cases} \label{jump_cond_1} \\
&&\beta^-\nabla\phi^{-}_T(F)\cdot\mathbf{ v}(F)=\beta^+\nabla\phi^{+}_T(F)\cdot\mathbf{ v}(F), \label{jump_cond_2}
\end{eqnarray}
where $F$ is an arbitrary point as described in Section \ref{sec:preliminaries}. Let $\overline{\mathcal{I}^s}=\mathcal{I}^s\cup\mathcal{I}^{int}$, $s=\pm$. Without loss of generality, we assume that $\abs{\overline{\mathcal{I}^-}} \leq \abs{\overline{\mathcal{I}^+}}$. For an IFE function $\phi_T$ under the integral degrees of freedom constraints
\begin{equation}
\label{shape_func_1}
\frac{1}{|b_i|}\int_{b_i}\phi_T(X)ds=v_i, \;\;\;\; i\in \mathcal{I},
\end{equation}
the condition \eqref{jump_cond_1} implies that $\phi_T$ can be written in the following form
\begin{equation}
\label{shape_func_2}
\phi_T(X) =
\begin{cases}
 \phi^{-}_T(X)  = \phi^{+}_T(X)+c_0L(X) & \text{if} \;\; X\in T_p^-, \\
 \phi^{+}_T(X)  = \sum_{i\in\overline{\mathcal{I}^-}}c_i\psi_{i,T}(X)+\sum_{i\in\mathcal{I}^+}v_i\psi_{i,T}(X)& \text{if} \;\; X\in T_p^+,
\end{cases}
\end{equation}
where $L(X)=\bar{\mathbf{ n}}\cdot(X-D)$ and $\nabla L(X) = \bar{\mathbf{n}}$. Then, applying the condition \eqref{jump_cond_2} to \eqref{shape_func_2} leads to
\begin{equation}
\label{eq_c0}
c_0=
k\left( \sum_{i\in\overline{\mathcal{I}^-}}c_i\nabla\psi_{i,T}(F)\cdot\mathbf{v}(F)+\sum_{i\in\mathcal{I}^+}v_i\nabla\psi_{i,T}(F)\cdot\mathbf{v}(F) \right),
\end{equation}
where $k = \left( \frac{1}{\rho}-1 \right)\frac{ 1}{\bar{\mathbf{n}}\cdot\mathbf{v}(F)}$ is well defined for $h$ small enough, since $\bar{\mathbf{ n}}\cdot\mathbf{v}(F)\geqslant1-Ch^2>0$ by Lemma 3.1 in \cite{2016GuoLin}. Moreover we have
\begin{equation}
\label{k_estimate}
\abs{k} \leq \abs{\left( \frac{1}{\rho}-1 \right)} \frac{1}{1 - Ch^2}.
\end{equation}
Substituting \eqref{eq_c0} into \eqref{shape_func_2}, seting \eqref{shape_func_1} for $j\in\overline{\mathcal{I}^-}$ and using the first property in \eqref{eqn_psi_iT_values_bounds} for $i,j\in\overline{\mathcal{I}^-}$, we obtain
\begin{equation*}
\begin{split}
\frac{1}{h}\int_{b_j}\phi_T(X)ds=&\frac{1}{h}\int_{b_j\cap T^-}\left(\phi^{+}_T(X)+c_0L(X)\right)ds+\frac{1}{h}\int_{b_j\cap T^+}\phi^{+}_T(X)ds\\
=&\frac{1}{h}\int_{b_j}\phi^{+}_T(X)ds+\frac{c_0}{h}\int_{b_j\cap T^-}L(X)ds\\
=&\sum_{i\in\overline{\mathcal{I}^-}}\left(\delta_{ij}+\frac{k}{h}\nabla\psi_{i,T}(F)\cdot\mathbf{v}(F)\int_{b_j\cap T^-}L(X)ds\right)c_i\\
+& \frac{k}{h}\int_{b_j\cap T^-}L(X)ds\sum_{i\in\mathcal{I}^+}\left(\nabla\psi_{i,T}(F)\cdot\mathbf{v}(F)\right)v_i=v_j, ~ j\in\overline{\mathcal{I}^-}
\end{split}
\end{equation*}
which can be written as a Sherman-Morrison system:
\begin{equation}
\label{linear_system}
(I+k\,\bfdelta \bfgamma^T )\bfc = \bfb,
\end{equation}
about the coefficients $\mathbf{c}=(c_i)_{i\in \overline{\mathcal{I}^-}}$, where
\begin{equation}
\label{eq_gammadelta}
\bfgamma=\left( \nabla\psi_{i,T}(F)\cdot\mathbf{v}(F) \right)_{i\in\overline{\mathcal{I}^-}}, \;\;\; \bfdelta=\frac{1}{h}\left( \int_{b_i\cap T^-}L(X)ds \right)_{i\in\overline{\mathcal{I}^-}},
\end{equation}
\begin{equation}
\label{eq_rhs}
\bfb=
\left(
v_i- \frac{k}{h}\int_{b_i\cap T^-}L(X)ds \sum_{j\in\mathcal{I}^+} \nabla\psi_{j,T}(F)\cdot\mathbf{v}(F)v_j
 \right)_{i\in\overline{\mathcal{I}^-}}
\end{equation}
are all column vectors.

Now we present two lemmas that are fundamental for the unisolvence of the IFE shape functions in the proposed form.

\begin{lemma}
\label{lemma_bound01}
For an interface element with arbitrary interface location and an arbitrary point $F\in l$, we have
$
\bfgamma^T\bfdelta \in [0,1]
$.
\end{lemma}
\begin{proof}
Because of the similarity, we only give the proof for the rectangular mesh. Without loss of generality, we consider the typical rectangle: $A_1=(0,0),A_2=(h,0),A_3=(h,h),A_4=(0,h)$ with $b_1=\overline{A_1A_2},b_2=\overline{A_2A_3},b_3=\overline{A_3A_4},b_4=\overline{A_4A_1}$. Taking into account of rotation, there are two possible interface elements that $\Gamma$ cuts $b_1$ and $b_4$ or cuts $b_1$ and $b_3$. For simplification, we only show the first case. And similar arguments apply to the second case. Let $D=(dh,0)$ and $E=(0,eh)$, for some $d,e\in[0,1]$ and $F=(td,(h-t)e)$, for some $t\in[0,h]$. Thus, $\bar{\mathbf{ n}}=(e,d)/\sqrt{d^2+e^2}$. By direct calculation, we have
$$
\bfgamma^T\bfdelta= \frac{de}{4(d^2+e^2)} \left( 5(d^2+e^2)+6(2t/h-1)(d^2e-de^2)-6de \right),
$$
which shows that $\bfgamma^T\bfdelta$ is linear function in terms of $t$. Furthermore, by a direct verification, we have
\begin{subequations}
\begin{align*}
&\bfgamma^T\bfdelta = \frac{de}{4(d^2+e^2)} \left( 5(d^2+e^2)-6(d^2e-de^2)-6de \right)\in[0,1],       &\text{if~}\;t=0, \\
&\bfgamma^T\bfdelta = \frac{de}{4(d^2+e^2)} \left( 5(d^2+e^2)+6(d^2e-de^2)-6de \right)\in[0,1],       &\text{if~}\;t=h,
\end{align*}
\end{subequations}
and these guarantee $\bfgamma^T\bfdelta \in [0,1]$.
\end{proof}

\begin{lemma}
\label{lemma_gammadelta}
For sufficiently small $h$, there exists a constant $C$ depending only on $\rho$ such that
\begin{equation}
\label{lemma_gammadelta_1}
1+k\,\bfgamma^T\bfdelta \geq \min{ \left(1, \frac{1}{\rho} \right) }-Ch.
\end{equation}
\end{lemma}
\begin{proof}
We first consider the case $F\in l$ so that $\mathbf{v}(F)=\overline{\mathbf{ n}}$, and thus, $k=1/\rho-1$. Then, by Lemma \ref{lemma_bound01}, we have $1+k\,\bfgamma^T\bfdelta \geq \min{ \left(1, \frac{1}{\rho} \right) }$, which implies \eqref{lemma_gammadelta_1} naturally.
For the case $F \in \Gamma \cap T$, we introduce an auxiliary vector $\bar{\bfgamma}=\left( \nabla\psi_{i,T}(F_{\bot})\cdot \bar{\bfn} \right)_{i\in\overline{\mathcal{I}^-}}$ where $F_{\bot}$ is the orthogonal projection of $F$ onto $l$. Then from the Lemma \ref{lemma_bound01}, we have $\bar{\bfgamma}^T\bfdelta\in[0,1]$. Therefore, the proof essentially follows the same argument as Lemma 3.1 in \cite{2016GuoLin}.
\end{proof}

\begin{thm}[$\mathbf{ Unisolvence}$]
\label{thm_unisol}
Let $\mathcal{T}_h$ be a mesh with $h$ small enough. Then, on every element $T\in\mathcal{T}^i_h$, given any vector $v=(v_1,v_2,v_3,v_4)\in \mathbb{R}^4$, there exists a unique IFE function $\phi_T$ in the form of \eqref{shape_func_2} satisfying approximated jump conditions \eqref{jump_cond_1}-\eqref{jump_cond_2}.
Furthermore, we have the following explicit formula for the coefficients in the IFE shape functions:
\begin{equation}
\label{solution}
\bfc = \bfb-k\frac{(\bfgamma^T \bfb)\bfdelta}{1+k\bfgamma^T\bfdelta}.
\end{equation}
\end{thm}
\begin{proof}
Lemma \ref{lemma_gammadelta} implies $1+k\bfgamma^T\bfdelta \neq 0$ for $h$ small enough. Hence, the existence and uniqueness for coefficients $c_i, i \in \overline{\mathcal{I}^-}$ and $c_0$ as well as formula \eqref{solution} follow straightforwardly from the well known properties of the \textit{Sherman-Morrison} formula and \eqref{eq_c0}.

~~~ 
\end{proof}


On each interface element $T$, Theorem \ref{thm_unisol} guarantees the existence and uniqueness of the IFE shape functions $\phi_{i,T}$, $i\in \mathcal{I}$  such that
\begin{equation}
\label{ife_shape_func}
\frac{1}{|b_j|}\int_{b_j}\phi_{i,T}(X)ds=\delta_{ij},~~i, j \in \mathcal{I},
\end{equation}
where $\delta_{ij}$ is the Kronecker delta function, which can be used to define the local IFE space as
\begin{equation}
\label{ife_space}
S_h^{int}(T)=\textrm{Span}\{ \phi_{i,T}:\;i\in \mathcal{I} \}.
\end{equation}

\commentout{
As in \cite{2013ZhangTHESIS}, depending on whether the interface cuts two adjacent or opposite edges, there are two types of interface elements.
According the location of $M_i, i \in \mathcal{I}$, a Type I element can be in $3$ cases and a Type II element can be in $2$ cases, see Figure \ref{fig_noncon_P_TypeI} and Figure \ref{fig_noncon_P_TypeII} for
illustrations of typical configurations.

 Let $F$ be an arbitrary point on $\Gamma\cap T$. We emphasize that, once chosen, $F$ should be fixed for each interface element.
\begin{figure}[H]
\centering
\includegraphics[width=6in]{figure5_1-eps-converted-to.pdf}
\caption{ \textbf{Type I} interface element and 3 cases }
\label{fig_noncon_P_TypeI}
\end{figure}

\begin{figure}[H]
\centering
\includegraphics[width=4in]{figure5_2-eps-converted-to.pdf}
\caption{ \textbf{Type II} interface element and 2 cases }
\label{fig_noncon_P_TypeII}
\end{figure}
}



As usual, the local IFE space can be employed to form a suitable global IFE function space on $\Omega$ in a finite element scheme.
For example, we can consider the following global IFE space:
\begin{equation}
\begin{split}
\label{ife_global_space}
S_h(\Omega)=& \left\{  v\in L^2(\Omega):v|_{T}\in S_h^{non}(T) \text{ if }T \in \mathcal{T}_h^n,v|_{T}\in S_h^{int}(T) \text{ if }T \in \mathcal{T}_h^i;
  \right.\\
& \int_{e} v|_{T_1}(P)ds(P) = \int_{e} v|_{T_2}(P)ds(P) ~\forall  e\in \mathcal{E}_h, \left. ~\forall\,T_1, T_2 \in \mathcal{T}_h \text{~such that~} e \in T_1\cap T_2  \right \}.
\end{split}
\end{equation}

\subsection{Properties of the IFE Shape Functions}

In this subsection, we present some fundamental properties for the IFE shape functions $\phi_T$. The first two results are very close to Theorem 5.2 and Lemma 5.3 in \cite{2016GuoLin} and the proofs of these results are essentially the same.

\begin{thm}[Bounds of IFE shape functions]
\label{bounds_shape_fun}
There exists a constant C, independent of interface location, such that
\begin{equation}
\label{boound_shape_func}
|\phi_{i,T}|_{k,\infty,T}\leqslant Ch^{-k}, ~~i\in \mathcal{I}, ~~k= 0, 1, 2, ~~\forall \,T \in \mathcal{T}^i_h.
\end{equation}
\end{thm}

\begin{lemma} [Partition of Unity]
\label{lem:POU}
On every interface element $T \in \mathcal{T}_h^i$, we have
\begin{equation}
\label{eq_POU_1}
\sum_{i\in \mathcal{I}}\phi_{i,T}(X)\equiv 1.
\end{equation}
\end{lemma}

Now, on every interface element $T$, for each $i \in \mathcal{I}$, we choose arbitrary points $\overline{X}_i \in l$ to construct two vector functions:
\begin{equation}
\label{identity}
\begin{split}
\bfLambda_s(X)=&\sum_{i\in\mathcal{I}}(M_i-X)\phi^{s}_{i,T}(X)+\sum_{i\in \mathcal{I}^{s'}}(\overline{M}^s(F) - I)^T(M_i-\overline{X}_i)\phi^{s}_{i,T}(X)\ \\
&+\frac{1}{h}\sum_{i\in\mathcal{I}^{int}}\int_{b_i\cap T^{s'}}(\overline{M}^s(F) - I)^T(P-\overline{X}_i)\phi^{s}_{i,T}(X)ds(P)  , \;\;\; \text{if}\; X\in T_p^s,
\end{split}
\end{equation}
where $s=\pm$ and $p=curve$ or $line$. By Lemma 3.4 in \cite{2016GuoLin},  $\bfLambda_s(X)$ is well defined since it is independent of location of $\overline{X}_i\in l, i \in \mathcal{I}$. We can simplify $\bfLambda_s(X)$ further by the partition of unity:
\begin{equation}
\label{identity_modified}
\begin{split}
   \bfLambda_s(X)=&\sum_{i\in \mathcal{I}}M_i\phi^{s}_{i,T}(X) - X + \sum_{i \in \mathcal{I}^{s'}}(\overline{M}^s(F) -I)^T(M_i-\overline{X}_i)\phi^{s}_{i,T}(X),    \\
   &+\frac{1}{h}\sum_{i\in\mathcal{I}^{int}}\int_{b_i\cap T^{s'}}(\overline{M}^s(F) - I)^T(P-\overline{X}_i)\phi^{s}_{i,T}(X)ds(P),
   \end{split}
  \end{equation}
from which we have $\bfLambda_s(X)\in \left[\Pi_T\right]^2$, since $\phi^{s}_{i,T}(X)\in \Pi_T, s = \pm$, for $i\in\mathcal{I}$. Moreover, by the independence of $\overline{X}_i$, $i\in\mathcal{I}$, we could interchange $\overline{X}_i$ with an arbitrary fixed point $\overline{X} \in l$ and obtain
\begin{equation}
\label{identity_compute1}
\begin{split}
\bfLambda_s(X)=&\sum_{i\in \mathcal{I}^s\cup\mathcal{I}^{int}}(M_i-\overline{X})\phi^{s}_{i,T} + \sum_{i \in \mathcal{I}^{s'}}(\overline{M}^s(F))^T(M_i-\overline{X})\phi^{s}_{i,T}
- X + \overline{X}\sum_{i\in\mathcal{I}}\phi^{s}_{i,T}\\
&+(\overline{M}^s(F) - I)^T  \sum_{i\in\mathcal{I}^{int}}\left(\frac{1}{h}\int_{b_i\cap T^{s'}} (P-\overline{X})ds(P) \right) \phi^{s}_{i,T}(X).
\end{split}
\end{equation}
By the identity $ \frac{1}{h}\int_{b_i\cap T^+} (P-\overline{X})ds(P)+ \frac{1}{h}\int_{b_i\cap T^-} (P-\overline{X})ds(P)=(M_i-\overline{X})$, $i\in\mathcal{I}^{int}$, \eqref{identity_compute1} yields
\begin{equation}
\label{identity_compute2}
\begin{split}
\bfLambda_s(X)
=&\sum_{i\in \mathcal{I}^s}(M_i-\overline{X})\phi^{s}_{i,T} + \sum_{i \in \mathcal{I}^{s'}\cup\mathcal{I}^{int}}(\overline{M}^s(F))^T (M_i-\overline{X})\phi^{s}_{i,T}  - p_0(X) \\
&-(\overline{M}^s(F) - I)^T \sum_{i\in\mathcal{I}^{int}}\left( \frac{1}{h}\int_{b_i\cap T^{s}} (P-\overline{X})ds(P)\right) \phi^{s}_{i,T}(X),
\end{split}
\end{equation}
where $\mathbf{p}_0(X)=X - \overline{X}\sum_{i\in\mathcal{I}}\phi^{s}_{i,T}=X-\overline{X}$, $s=\pm$, by the partition of unity. We consider a vector function
\begin{equation}
\label{eq_psi_x}
\bfpsi_0(X)=\left\{
\begin{aligned}
&\bfpsi_0^{+}(X)= \mathbf{p}_0(X), & \text{if} ~ X\in T_p^{+}, \\
&\bfpsi_0^-(X)=(\overline{M}^{+}(F))^T\mathbf{p}_0(X), & \text{if} ~ X\in T_p^-,
\end{aligned}
\right.
\end{equation}
where $p=curve$ or $line$ and $\overline{X}$ is an arbitrary point fixed on $l$.
\begin{lemma}
\label{lemma_X}
For any point $\overline{X}\in l$, the vector function $\bfpsi_0$ defined by \eqref{eq_psi_x} belongs to $\left[ S_h^{int}(T) \right]^2$.
\end{lemma}
\begin{proof}
It suffices to verify that $\bfpsi_0$ satisfies the conditions \eqref{jump_cond_1} and \eqref{jump_cond_2}. First it is easy to see $\partial_{xx}(\bfpsi_0^+)=\partial_{xx}(\bfpsi_0^-)=\mathbf{ 0}$, $s=\pm$. Besides, for any $\overline{X}'\in l$, Lemma 3.3 in \cite{2016GuoLin} implies that $\bfpsi^-_0(\overline{X}')-\bfpsi^+_0(\overline{X}')=\left(\overline{M}^{+}(F)-I\right)^T \left(\overline{X}'-\overline{X}\right)=0$, and hence $\bfpsi_0$ satisfies \eqref{jump_cond_1}. Finally, Lemma 3.3 in \cite{2016GuoLin} also shows that
\begin{equation*}
\beta^{-}\nabla \bfpsi_0^{-}(F)\cdot\mathbf{ v}(F)=\beta^{-}\left(\overline{M}^{+}(F)\right)^T\mathbf{ v}(F)=\beta^+\mathbf{ v}(F)=\beta^+\nabla\bfpsi_0^+(F)\cdot \mathbf{ v}(F).
\end{equation*}
Therefore $\bfpsi_0$ satisfies \eqref{jump_cond_2}.
\end{proof}

Now we consider an auxiliary piecewise defined vector function given by
\begin{equation}
\label{auxi_fun}
\bfLambda(X)=
\left\{
\begin{aligned}
&\bfLambda^+(X)=\bfLambda_+(X) & \text{if} \; X\in T_p^+, \\
&\bfLambda^-(X)=(\overline{M}^+(F))^T\bfLambda_-(X) & \text{if} \; X\in T_p^-,
\end{aligned}
\right.
\end{equation}
where $p=curve$ or $line$.
\begin{thm}
\label{lambda_in_SI}
$\bfLambda(X)$ defined by \eqref{auxi_fun} is in $\left[ S_h^{int}(T) \right]^2$ and
\begin{equation}
\label{edge0}
\int_{b_i}\bfLambda(X)ds(X)=\mathbf{0}, ~ \forall i\in\mathcal{I}.
\end{equation}
\end{thm}
\begin{proof}
First, by comparing the coefficients of $\phi^{s}_{i,T}$ in $\bfLambda_s$ in \eqref{identity_compute1} for $s=+$ and $\phi^{s}_{i,T}$ in $\bfLambda_s$ in \eqref{identity_compute2} for $s=-$ and using Lemma \ref{lemma_X}, we have
\begin{equation}
\label{linear_combin}
\begin{split}
\bfLambda=&\sum_{j\in \mathcal{I}^+\cup\mathcal{I}^{int}}(M_i-\overline{X})\phi_{j,T} + \sum_{j \in \mathcal{I}^{-}}(\overline{M}^+(F))^T(M_j-\overline{X})\phi_{j,T}
- \bfpsi_0\\
&+(\overline{M}^+(F) - I)^T  \sum_{j\in\mathcal{I}^{int}}\left(\frac{1}{h}\int_{b_j\cap T^{-}} (P-\overline{X})ds(P) \right) \phi_{j,T}
\end{split}
\end{equation}
which is actually a linear combination of $(\phi_{j,T},0)^T$, $(0,\phi_{j,T})^T$, and $\bfpsi_0$. Therefore $\bfLambda\in \left[ S_h^{int}(T) \right]^2$. Next for $i\in\mathcal{I}^s$, $s=\pm$, it is easy to show \eqref{edge0}. And for $i\in\mathcal{I}^{int}$,  by \eqref{linear_combin}, we have
\begin{equation*}
\begin{split}
\int_{b_i}\bfLambda(X)ds(X)=&(M_i-\overline{X}) - \frac{1}{h}\int_{b_i\cap T^+} (X-\overline{X}) ds(X) - (\overline{M}^+(F))^T\frac{1}{h}\int_{b_i\cap T^-} (X-\overline{X}) ds(X) \\
& + (\overline{M}^+(F) - I)^T \left(\frac{1}{h}\int_{b_i\cap T^{-}} (P-\overline{X})ds(P) \right) =\mathbf{0}.
\end{split}
\end{equation*}
~~~
\end{proof}
\begin{thm}
\label{thm_identity0}
On every interface element $T \in \mathcal{T}_h^i$ we have
\begin{equation}
\label{identity0}
\begin{split}
&\sum_{i\in\mathcal{I}}(M_i-X)\phi^{s}_{i,T}(X)+\sum_{i\in \mathcal{I}^{s'}}(\overline{M}^s(F) - I)^T(M_i-\overline{X}_i)\phi^{s}_{i,T}(X)\ \\
&+\frac{1}{h}\sum_{i\in\mathcal{I}^{int}}\int_{b_i\cap T^{s'}}(\overline{M}^s(F) - I)^T(P-\overline{X}_i)\phi^{s}_{i,T}(X)ds(P)=\mathbf{ 0}  , \;\;\; \forall X\in T_p^s,
\end{split}
\end{equation}
and
\begin{equation}
\label{identity1}
\begin{split}
&\sum_{i\in\mathcal{I}}(M_i-X)\partial_d\phi^{s}_{i,T}(X)+\sum_{i\in \mathcal{I}^{s'}}\left[ (\overline{M}^-(F) - I)^T(M_i-\overline{X}_i)\partial_d\phi^{s}_{i,T}(X) \right] \\
&+\sum_{i\in\mathcal{I}^{int}}\left(\frac{1}{h}\int_{b_i\cap T^s}(\overline{M}^s(F) - I)^T(P-\overline{X}_i)ds(P)\right)\phi^{s}_{i,T}(X) - \mathbf{ e}_d = \mathbf{ 0}, ~ \forall X \in T_p^s,
\end{split}
\end{equation}
where $s=\pm$, $p=curve$ or $line$ and $d = 1, 2$, $\partial_1 = \partial_x, \partial_2 = \partial_y$ are partial differential operators, and $\mathbf{ e}_d, d = 1, 2$ is the canonical $d$-th unit vector in
$\mathbb{R}^2$.
\end{thm}
\begin{proof}
The identity \eqref{identity0} follows from Theorem \ref{lambda_in_SI} and the unisolvence, and \eqref{identity1} is the derivative of \eqref{identity0}.
\end{proof}

\section{Optimal Approximation Capabilities of IFE Spaces}

In this section, we show the optimal approximation capabilities for two classes of IFE spaces defined by curved interface and its line approximation, respectively. This is achieved by
deriving error bounds for the interpolation in IFE spaces.

We start from the local interpolation operator $I_{h,T}: C^0(T) \rightarrow S_h(T)$ on an element $T \in \mathcal{T}_h$:
\begin{equation}
\label{local_interp}
I_{h,T}u(X)= \begin{cases}
\sum_{i\in\mathcal{I}}\left(\frac{1}{|b_i|}\int_{b_i}u(P)ds(P)\right) \, \psi_{i,T}(X), & \text{if~} T \in \mathcal{T}_h^n, \vspace{1mm}\\
\sum_{i\in\mathcal{I}}\left(\frac{1}{|b_i|}\int_{b_i}u(P)ds(P)\right) \, \phi_{i,T}(X), & \text{if~} T \in \mathcal{T}_h^i.
\end{cases}
\end{equation}
Then, as usual, the global IFE interpolation $I_{h}: C^0(\Omega)\rightarrow S_h(\Omega)$ can be defined piecewisely:
\begin{equation}
\label{global_interp}
(I_hu)|_{T}=I_{h,T}u, \;\;\; \forall T\in \mathcal{T}_h.
\end{equation}

First for the local interpolation $I_{h,T}u$ on every non-interface element $T\in\mathcal{T}^n_h$, the standard argument \cite{1992RanacherTurek} yields
\begin{equation}
\label{noninterf_estimate}
\| I_{h,T}u-u\|_{0,T}+h| I_{h,T}u-u|_{1,T}\leqslant Ch^2|u|_{2,T}, ~\forall u \in H^2(T).
\end{equation}

On each interface element $T \in \mathcal{T}_h^i$, for $s=\pm$, $i\in\mathcal{I}$, we consider two functions
$E_i \;:\; b_i\times T_{non} \rightarrow \mathbb{R} $ and $F_i \;:\; b_i\times T_{non} \rightarrow \mathbb{R} $ such that
\begin{equation}
\begin{split}
\label{eq_EiFi_Is}
& E_i(P,X)= ( ( M^s(\widetilde{Y}_i) - \overline{M}^s(F) )\nabla u^s(X) )\cdot(P-\overline{X}_i),  ~ \text{if}~P\in b_i\cap T^{s'}, X\in T_{non}^s, \\
& F_i(P,X)= \left( (\overline{M}^s(F) - I)\nabla u^s(X) \right) \cdot (\widetilde{Y}_i-\overline{X}_i),  ~ \text{if}~P\in b_i\cap T^{s'}, X\in T_{non}^s,
\end{split}
\end{equation}
where $\widetilde{Y}_i=\widetilde{Y}_i(P,X)$ and $\overline{X}_i \in l, i \in \mathcal{I}$.
We note that $E_i$ and $F_i$ are piecewisely defined on $b_i\times T_{non}$.
Furthermore, integrating \commentout{ for $i\in\mathcal{I}^{s'}\cup\mathcal{I}^{int}$,} \eqref{eq_EiFi_Is} leads to the following two functions $\mathcal{E}_i:T_{non}\rightarrow \mathbb{R}$ and $\mathcal{F}_i:T_{non} \rightarrow \mathbb{R}$:
\begin{equation}
\label{EF_intgral}
\mathcal{E}_i(X)=\frac{1}{h}\int_{b_i\cap T^{s'}}E_i(P,X)ds(P), ~~ \mathcal{F}_i(X)=\frac{1}{h}\int_{b_i\cap T^{s'}}F_i(P,X)ds(P), ~~ \textit{if}~X\in T_{non}^s. 
\end{equation}
Note that $\mathcal{E}_i$ and $\mathcal{F}_i$ are also piecewisely defined on $T$. Their estimates are given in the following theorem.

\begin{lemma}
\label{lemma_EF_integra_estimate}
There exists a constant $C>0$ independent of the interface location such that the following estimates hold for
every $T\in\mathcal{T}^i_h$ and $u\in PC^2_{int}(T)$:
\begin{align}
\label{eq_lemma_EF_integra_estimate}
\| \mathcal{E}_i\|_{0,T_{non}^s}\leqslant Ch^2 |u|_{1,T^s}, ~~~~ \| \mathcal{F}_i\|_{0,T_{non}^s}\leqslant Ch^2 |u|_{1,T^s} \commentout{ \;\;\;\; i\in \mathcal{I}^{s'}\cup\mathcal{I}^{int}}, ~~s=\pm.
\end{align}
\end{lemma}
\begin{proof}
By the Lemma 5.7 in \cite{2016GuoLin}, for fixed $P\in b_i\cap T^{s'}$, we have
\begin{equation}
\label{EF_estimate}
\| E_i(P,\cdot) \|_{0,T_{non}^s} \leqslant Ch^2|u|_{1,T^s}, ~~\| F_i(P,\cdot) \|_{0,T_{non}^s} \leqslant Ch^2|u|_{1,T^s} .
\end{equation}
Then, the estimate \eqref{eq_lemma_EF_integra_estimate} follows from the same arguments as in the proof for Lemma \ref{reminder_estimate}.
\end{proof}


We now derive expansions for the interpolation error. The first group of expansions are for the interpolation error
at $X \in T_{non}$ given in the following theorem.

\begin{thm}
\label{thm_expan}
Let $T \in \mathcal{T}_h^i$ and $u\in PC_{int}^2(T)$. Then for any $\overline{X}_i\in l, i \in \mathcal{I}$, we have
\begin{subequations}\label{eq4_35}
\begin{equation}
\label{thm_expan_1}
I_{h,T}u(X)-u(X)=\sum_{i\in\mathcal{I}^{s'}\cup\mathcal{I}^{int}}(\mathcal{E}_i+\mathcal{F}_i)\phi_{i,T}(X)+\sum_{i\in\mathcal{I}}\mathcal{R}_i\phi_{i,T}(X),~\forall X \in T_{non}^s\cap T^s_p, ~s = \pm,
\end{equation}
\begin{equation}
\label{thm_expan_2}
\partial_d( I_{h,T}u(X)-u(X))=\sum_{i\in\mathcal{I}^{s'}\cup\mathcal{I}^{int}}(\mathcal{E}_i + \mathcal{F}_i )\partial_d\phi_{i,T}(X)+\sum_{i\in\mathcal{I}}\mathcal{R}_i \partial_d\phi_{i,T}(X), ~\forall X \in T_{non}^s\cap T^s_p, ~s = \pm,
\end{equation}
\end{subequations}
where $p=curve$ or $line$, $d=1$ or $2$, $\mathcal{R}_i^s$ and $\mathcal{E}_i$, $\mathcal{F}_i$ are given by \eqref{reminder_integral}, and \eqref{EF_intgral}, respectively.
\end{thm}

\begin{proof}
First, for $X \in T_{non}^s\cap T^s_p, s = \pm$, substituting the expansion \eqref{eq_integral_expan_1}, \eqref{eq_integral_expan_2} and \eqref{eq_integral_expan_3} into the IFE interpolation \eqref{local_interp} and using the partition of unity yields
\begin{equation}
\begin{split}
\label{eq_expan_1}
I_{h,T}u(X) =& u^s(X)+ \nabla u^s(X)\cdot \sum_{i\in\mathcal{I}}(M_i-X)\phi^{s}_{i,T}(X)+ \sum_{i\in\mathcal{I}}R_i^s\phi^{s}_{i,T}(X)\\
&+ \sum_{i\in \mathcal{I}^{s'}}\frac{1}{h}\int_{b_i}\left( \left( M^s(\widetilde{Y}_i)-I \right)\nabla u^s(X) \right)\cdot (P-\widetilde{Y}_i)ds(P)\phi^{s}_{i,T}(X) \\
 &+\sum_{i\in \mathcal{I}^{int}}\frac{1}{h}\int_{b_i\cap T^{s'}}\left( \left( M^s(\widetilde{Y}_i)-I \right)\nabla u^s(X) \right)\cdot (P-\widetilde{Y}_i)ds(P)\phi^{s}_{i,T}(X), ~s = \pm.
\end{split}
\end{equation}
Applying \eqref{identity0} in Theorem \ref{thm_identity0}, we have
\begin{equation}
\begin{split}
\label{eq_expan_2}
I_{h,T}u(X) =& u^s(X)- \sum_{i\in\mathcal{I}^{s'}}\frac{1}{h}\int_{b_i}\left( \left(\overline{M}^s(F) - I\right)\nabla u^s(X) \right) \cdot(P-\overline{X}_i)ds(P)\phi_{i,T}(X) + \sum_{i\in\mathcal{I}}R_i^s\phi_{i,T}(X) \\
&-\sum_{i\in\mathcal{I}^{int}} \frac{1}{h} \int_{b_i\cap T^{s'}}\left( \left(\overline{M}^s(F) - I \right)\nabla u^s(X) \right)\cdot(P-\overline{X}_i)ds(P) \phi^{s}_{i,T}(X) \\
&+ \sum_{i\in \mathcal{I}^{s'}}\frac{1}{h}\int_{b_i}\left( \left( M^s(\widetilde{Y}_i)-I \right)\nabla u^s(X) \right)\cdot (P-\widetilde{Y}_i)ds(P)\phi_{i,T}(X) \\
 &+\sum_{i\in \mathcal{I}^{int}}\frac{1}{h}\int_{b_i\cap T^{s'}}\left( \left( M^s(\widetilde{Y}_i)-I \right)\nabla u^s(X) \right)\cdot (P-\widetilde{Y}_i)ds(P)\phi_{i,T}(X), ~s = \pm.
\end{split}
\end{equation}
Then substituting $P-\overline{X}_i=(P- \widetilde{Y}_i)+(\widetilde{Y}_i -\overline{X}_i)$ into \eqref{eq_expan_2} yields \eqref{thm_expan_1}. Furthermore, applying the expansions \eqref{eq_integral_expan_1}, \eqref{eq_integral_expan_2}
in $\partial_d I_{h,T}u(X)=\sum_{i\in\mathcal{I}}\frac{1}{h}\int_{b_i}u(P)ds(P) \partial_d\phi_{i,T}(X), d = 1, 2$, yields
\begin{equation}
\label{eq_expan_3}
\begin{split}
\partial_d I_{h,T}u(X) =&  \nabla u^s(X)\cdot \sum_{i\in\mathcal{I}}(M_i-X)  \partial_d \phi^{s}_{i,T}(X)+ \sum_{i\in\mathcal{I}}R_i^s  \partial_d \phi^{s}_{i,T}(X)\\
&+ \sum_{i\in \mathcal{I}^{s'}}\frac{1}{h}\int_{b_i}\left( \left( M^s(\widetilde{Y}_i)-I \right)\nabla u^s(X) \right)\cdot (P-\widetilde{Y}_i)ds(P)  \partial_d \phi^{s}_{i,T}(X) \\
 &+\sum_{i\in \mathcal{I}^{int}}\frac{1}{h}\int_{b_i\cap T^{s'}}\left( \left( M^s(\widetilde{Y}_i)-I \right)\nabla u^s(X) \right)\cdot (P-\widetilde{Y}_i)ds(P)  \partial_d \phi^{s}_{i,T}(X), ~s = \pm.
 \end{split}
\end{equation}
Finally, using \eqref{identity1} and similar argument above, we have \eqref{thm_expan_2}.
\end{proof}

The second group of expansions are for $X\in T_{int}$ which is much simpler. Using \eqref{1st_multiedge_1} in $I_{h,T}u(X)$ defined in \eqref{local_interp} and the partition of unity, we have
\begin{subequations}\label{eq4_35_2}
\begin{equation}
\label{Iint_multiedge_1}
I_{h,T}u(X)-u(X) = \sum_{i\in\mathcal{I}}\mathcal{R}_i \phi_{i,T}(X), ~~~~~~~ \forall~ X\in T_{int},
\end{equation}
\begin{equation}
\label{Iint_multiedge_2}
\partial_dI_{h,T}u(X)-\partial_du(X) = -\partial_d u(X) + \sum_{i\in\mathcal{I}}\mathcal{R}_i \partial_d\phi_{i,T}(X),~~~~~ \forall~ X\in T_{int}, ~~d=x,~y.
\end{equation}
\end{subequations}


\subsection{Curve Partition}
\label{sec:CurvePartition}
In this subsection, we derive error bounds for the interpolation in the IFE space defined according to the actual interface $\Gamma$ on each interface element, i.e.,
the local IFE functions on each interface element $T$ are defined by \eqref{ife_piecewise} with $T_p^s = T^s_{curve}$, $s=\pm$. We first derive an estimate for the IFE interpolation error on $T_{non}$.
\begin{thm}
\label{thm_interface_element_estimate_curve}
There exists a constant $C>0$ independent of the interface location such that for every $u\in PH^2_{int}(T)$
it holds
\begin{equation}
\label{interface_element_estimate_curve}
\| I_{h,T}u-u\|_{0,T_{non}} + h| I_{h,T}u-u |_{1,T_{non}}\leqslant Ch^2 (|u|_{1,T}+|u|_{2,T}), ~~\forall~T\in\mathcal{T}^i_h.
\end{equation}
\end{thm}
\begin{proof}
On each $T \in \mathcal{T}_h^i$, Theorem \ref{bounds_shape_fun} and Theorem \ref{thm_expan} show that there exists a constant $C$ such that for every $u\in PC^2_{int}(T)$ we have
\begin{equation}
\label{}
\| I_{h,T}u-u\|_{0,T_{non}^s}\leqslant C\left(\sum_{i \in\mathcal{I}^{s'}\cup\mathcal{I}^{int}}\left(\|\mathcal{E}_i\|_{0,T_{non}^s}+\|\mathcal{F}_i\|_{0,T_{non}^s}\right)+\sum_{i\in \mathcal{I}}\|\mathcal{R}_i\|_{0,T_{non}^s}\right),
\end{equation}
\begin{equation}
\| \partial_d(I_{h,T}u-u)\|_{0,T_{non}^s}\leqslant \frac{C}{h}\left(\sum_{i\in\mathcal{I}^{s'}\cup\mathcal{I}^{int}}\left(\| \mathcal{E}_i\|_{0,T_{non}^s}+\|\mathcal{F}_i\|_{0,T_{non}^s}\right)+\sum_{i\in \mathcal{I}}\| \mathcal{R}_i\|_{0,T_{non}^s}\right),
~d = 1, 2.
\end{equation}
Then, applying Lemmas \ref{lemma_EF_integra_estimate} and \ref{reminder_estimate} to the two estimates above yields
\begin{eqnarray*}
\| I_{h,T}u-u\|_{0,T_{non}^s} + h| I_{h,T}u-u |_{1,T_{non}^s}\leqslant Ch^2 (|u|_{1,T}+|u|_{2,T}), ~s = \pm.
\end{eqnarray*}
The estimate \eqref{interface_element_estimate_curve} for $u\in PC_{int}^2(T)$ follows from summing the inequality above for $s=-,+$. And the estimate \eqref{interface_element_estimate_curve} for $u\in PH_{int}^2(T)$ follows from the density hypothesis \textbf{(H4)}.
\end{proof}

Furthermore, for the estimation on $T_{int}$, we have
\begin{thm}
\label{thm_est_Iint}
There exists a constant $C>0$ independent of the interface location such that for every $u\in PH^2_{int}(T)$ it holds
\begin{equation}
\label{thm_est_Iint_eq_1}
\| I_{h,T}u-u\|_{0,T_{int}} + h| I_{h,T}u-u |_{1,T_{int}}\leqslant Ch^2 \|u\|_{1,6,T}, ~~\forall~T\in\mathcal{T}^i_h.
\end{equation}
\end{thm}
\begin{proof}
First Theorem \ref{bounds_shape_fun} and \eqref{1st_multiedge_est} imply that
\begin{equation*}
\| \mathcal{R}_i ~ \phi_{i,T} \|_{0,T_{int}}\leqslant Ch^2 \|u\|_{1,6,T} ~~~ \textrm{and} ~~~ \| \mathcal{R}_i ~ \partial_d\phi_{i,T} \|_{0,T_{int}}\leqslant Ch \|u\|_{1,6,T}.
\end{equation*}
where $d=x,y$. Using the H\"older's inequality again, we have
\begin{equation*}
\begin{split}
\left( \int_{T_{int}} (\partial_du)^2 dX \right)^{1/2}\leqslant \left( \int_{T_{int}}1^{3/2}dX\right)^{1/3} \left( \int_{T_{int}}(\partial_du)^{6}dX\right)^{1/6} \leqslant Ch \|u\|_{1,6,T}.
\end{split}
\end{equation*}
Then, \eqref{thm_est_Iint_eq_1} follows from applying estimates above together with the density hypothesis \textbf{(H4)} to expansions in \eqref{eq4_35}.
\end{proof}

Finally we can prove the following global estimate for the IFE interpolation by summing the local estimate over all the elements.
\begin{thm}
\label{thm_global_estimate}
For any $u\in PH^2_{int}(\Omega)$, the following estimate of interpolation error holds
\begin{equation}
\label{eq_global_estimate}
\| I_{h}u-u\|_{0,\Omega} + h| I_{h}u-u |_{1,\Omega}\leqslant Ch^2 \|u\|_{2,\Omega}.
\end{equation}
\begin{proof}
Putting \eqref{interface_element_estimate_curve} and \eqref{thm_est_Iint_eq_1} together, we have
\begin{equation}
\label{thm_est_on_interf_elem}
\| I_{h,T}u-u\|_{0,T} + h| I_{h,T}u-u |_{1,T}\leqslant Ch^2 (\|u\|_{2,T} + \|u\|_{1,6,T}), ~~\forall~T\in\mathcal{T}^i_h.
\end{equation}
Then, by summing \eqref{thm_est_on_interf_elem} and \eqref{noninterf_estimate} over all the interface and non-interface elements, we have
\begin{equation*}
\| I_{h}u-u\|_{0,\Omega} + h| I_{h}u-u |_{1,\Omega}\leqslant Ch^2 (\|u\|_{2,\Omega}+\|u\|_{1,6,\Omega}).
\end{equation*}
We note the following estimate from \cite{RenWei1994} that for any $p\geqslant2$
\begin{equation*}
\| u \|^2_{1,p,\Omega} \leqslant Cp\|u\|^2_{2,\Omega}.
\end{equation*}
Similar argument is used in \cite{2004LiLinLinRogers}. Combining the two inequalities above leads to \eqref{eq_global_estimate}.
\end{proof}
\end{thm}

\subsection{Line Partition}

In this subsection, we derive error bounds for the interpolation in the IFE space constructed by using the straight line to approximate the actual interface $\Gamma$ on each interface element, i.e., the local IFE functions on each interface element $T$ are defined by \eqref{ife_piecewise} with $T_p^s = T^s_{line}$, $s=\pm$.
Let $\overline{T}^s=T_{non}^s\cap T^s_{line}$ and $\widetilde{T}$ be the subset of $T$ sandwiched between $\Gamma$ and $l$ . Because
$\overline{T}^s \subseteq T_{non}$, by the same arguments for Theorem \ref{thm_interface_element_estimate_curve}, we have
\begin{equation}
\label{interface_element_estimate_line1}
\| I_{h,T}u-u\|_{0,\overline{T}^s} + h| I_{h,T}u-u |_{1,\overline{T}^s} \leqslant Ch^2 (|u|_{1,T}+|u|_{2,T}), ~~ s=\pm,  ~\forall~T\in\mathcal{T}^i_h.
\end{equation}
Similarly, the estimate \eqref{thm_est_Iint_eq_1} is also valid for the IFE space constructed using the straight line to approximate the actual interface $\Gamma$.

For $\widetilde{T}$, we note that there exists a constant $C$ such that $|\widetilde{T}|\leqslant Ch^3$. Then, applying the same arguments as those for
Theorem \ref{thm_est_Iint}, we can prove the following theorem.
\begin{thm}
\label{thm_interface_element_estimate_line2}
There exists a constant $C>0$ independent of the interface location such that for every $u\in PH^2_{int}(T)$ it holds
\begin{equation}
\label{interface_element_estimate_line2}
\| I_{h,T}u-u\|_{0,\widetilde{T}} + h| I_{h,T}u-u |_{1,\widetilde{T}} \leqslant Ch^2\|u\|_{1,6,T}, ~~ s=\pm,  ~\forall~T\in\mathcal{T}^i_h.
\end{equation}
\end{thm}

For each interface element $T \in \mathcal{T}_h^i$, because
\begin{eqnarray*}
T = \big(\overline{T}^- \cup \overline{T}^+ \cup \widetilde{T} \cup T_{int}\big)
\end{eqnarray*}
we can put estimates above together to have
\begin{equation}
\label{thm_est_on_interf_elem_line}
\| I_{h,T}u-u\|_{0,T} + h| I_{h,T}u-u |_{1,T}\leqslant Ch^2 (\|u\|_{2,T} + \|u\|_{1,6,T}), ~~\forall~T\in\mathcal{T}^i_h.
\end{equation}

%
%

Finally, by the same arguments for Theorem \ref{thm_global_estimate}, we can derive the global interpolation error estimate given in the following theorem for for the IFE space constructed by using the straight line to approximate the actual interface $\Gamma$.

\begin{thm}
\label{estimate_omega}
For any $u\in PH^2_{int}(\Omega)$, the following estimation of interpolation error holds
\begin{equation}
\label{eq_estimate_omega_1}
\| I_{h,T}u-u\|_{0,\Omega} + h| I_{h,T}u-u |_{1,\Omega}\leqslant Ch^2 \|u\|_{2,\Omega}.
\end{equation}
\end{thm}

\noindent
{\bf Remark}: The estimate \eqref{eq_estimate_omega_1} is also derived in Theorem 3.12 of \cite{2013ZhangTHESIS} through an argument based on the interpolation error
bounds for the rotated-$Q_1$ IFE space with the Lagrange type degrees of freedom.

\section{Numerical Examples}

In this section we present some numerical results to demonstrate the features of IFE interpolation and IFE solutions for IFE spaces. The interface problem that we tested is the same as the one used in \cite{2015LinLinZhang}. Specifically, the solution domain is $\Omega = (-1,1)\times (-1,1)$ which is divided into two subdomains $\Omega^-$ and $\Omega^+$ by a circular interface $\Gamma$ with radius $r_0=\pi/6.28$ such that
$
\Omega^-=\{ (x,y): x^2+y^2\leqslant r_0^2 \}.
$
Functions $f$ and $g$ in \eqref{eq1_1} are given such that the exact solution to interface problem described by \eqref{eq1_1}-\eqref{eq1_4} is given by the following
formula:
\begin{equation}
\label{example_u}
u(x,y) =\left\{
\begin{aligned}
&\frac{1}{\beta^-}r^{\alpha},  &\; (x,y)\in \Omega^-,\\
&\frac{1}{\beta^+}r^{\alpha}+\left( \frac{1}{\beta^-} -\frac{1}{\beta^+} \right)r_0^{\alpha}, & \; (x,y) \in \Omega^+,
\end{aligned}
\right.
\end{equation}
where $r=\sqrt{x^2+y^2}$ and $\alpha=5$. 

In our numerical experiment, we construct IFE spaces by rotated-$Q_1$ polynomials defined with the actual curve interface, and the flux continuity \eqref{jump_cond_2} is enforced at the midpoint $F$ of the curve $\Gamma\cap T$ for constructing IFE basis functions. To avoid redundancy, we only present numerical result of relatively large coefficient jump, i.e, $\beta^-=1$, $\beta^+=10000$. Similar behavior are observed for the reverse of jump values $\beta^-=10000$ and $\beta^+=1$, and for some small coefficient jumps. 

Since IFE functions on each interface element $T \in \mathcal{T}_h^i$ are defined as piecewise rotated-$Q_1$ polynomials by two subelements sharing a curved boundary $\Gamma \cap T$, integrations over these curve subelements require
special attentions when assembling the local matrix and vector. These two subelements can be such that one is a curved triangle and other one is a curved pentagon,
or they are two curved quadrilaterals, all of them just have one curved edge. For the quadratures on the curved pentagon, we can partition it further into a straight edge triangle and a curved edge quadrilateral. Then we use the standard isoparametric mapping for integrations on curved triangles and quadrilaterals.

Table \ref{table:bilinearIterpolationError_1_10000} and Table \ref{table:bilinearIFESolutionError_SymPen1_10000} present interpolation errors $u - I_hu$ and Galerkin IFE solution errors $u - u_h$, respectively, in terms of the $L^2$ and the semi-$H^1$ norms generated over a sequence of meshes with size $h$ from 1/20 to 1/1280. The rates listed in these tables are estimated by the numerical results generated on two consecutive meshes.

Data in Table \ref{table:bilinearIterpolationError_1_10000} clearly shows the optimal convergence rate for the IFE interpolation. The IFE solution
$u_h$ in Table \ref{table:bilinearIFESolutionError_SymPen1_10000} is generated by the standard Galerkin formulation with a discrete
bilinear form \cite{2004LiLinLinRogers,2003LiLinWu} without any penalties on interface edges used in the partially penalized IFE methods in \cite{2015LinLinZhang}.
This means the IFE method used to generate data in Table \ref{table:bilinearIFESolutionError_SymPen1_10000} is simpler than the one discussed in \cite{2015LinLinZhang}.
The data in Table \ref{table:bilinearIFESolutionError_SymPen1_10000} demonstrates that the classic scheme using the nonconforming IFE spaces developed in this article
performs optimally. Finally, we refer readers to \cite{2015LinSheenZhang, 2013ZhangTHESIS} for numerical results generated with the nonconforming IFE spaces defined with the line approximation.



\begin{table}[H]
\begin{center}
\begin{tabular}{|c |c c|c c|}
\hline
$h$   & $\|u - I_hu\|_{0,\Omega}$ & rate   & $|u - I_hu|_{1,\Omega}$ & rate   \\ \hline
1/20  & 6.3804E-4                 &        & 2.7693E-2               &        \\ \hline
1/40  & 1.6776E-4                 & 1.9272 & 1.4436E-2               & 0.9399 \\ \hline
1/80  & 4.3557E-5                 & 1.9454 & 7.4385E-3               & 0.9566 \\ \hline
1/160  & 1.1100E-5                 & 1.9723 & 3.7803E-3               & 0.9765 \\ \hline
1/320 & 2.8083E-6                 & 1.9828 & 1.9060E-3               & 0.9880 \\ \hline
1/640 & 7.0568E-7                 & 1.9926 & 9.5704E-4               & 0.9939 \\ \hline
1/1280 & 1.7692E-7                 & 1.9959 & 4.7959E-4               & 0.9968 \\ \hline
\end{tabular}
\end{center}
\caption{Interpolation errors and rates for the rotated-$Q_1$ IFE function, $\beta^-=1$ and $\beta^+=10000$.}
\label{table:bilinearIterpolationError_1_10000}
\end{table}


\begin{table}[H]
\begin{center}
\begin{tabular}{| c |c c|c c|}\hline
 $h$   & $\|u-u_h\|_{0,\Omega}$ & rate   & $|u-u_h|_{1,\Omega}$    & rate    \\ \hline
1/20  & 1.4221E-3              &        & 2.8852E-2               &         \\ \hline
1/40  & 3.4863E-4              & 2.0283 & 1.4822E-2               & 0.9610  \\ \hline
1/80  & 8.5873E-5              & 2.0214 & 7.5721E-3               & 0.9689  \\ \hline
1/160  & 2.1046E-5              & 2.0286 & 3.8057E-3               & 0.9925  \\ \hline
1/320 & 5.7133E-6              & 1.8812 & 1.9154E-3               & 0.9905  \\ \hline
1/640 & 1.4044E-6              & 2.0243 & 9.5891E-4               & 0.9982  \\ \hline
1/1280 & 3.4603E-7              & 2.0210 & 4.8004E-4               & 0.9982  \\ \hline
\end{tabular}
\end{center}
\caption{Galerkin solution errors and rates for the rotated-$Q_1$ IFE solution, $\beta^-=1$, $\beta^+=10000$.}
\label{table:bilinearIFESolutionError_SymPen1_10000}
\end{table}

\bibliographystyle{abbrv}

\begin{thebibliography}{10}

\bibitem{1998ChenZou}
Z.~Chen and J.~Zou.
\newblock Finite element methods and their convergence for elliptic and
  parabolic interface problems.
\newblock {\em Numer. Math.}, 79(2):175--202, 1998.

\bibitem{2007GongLiLi}
Y.~Gong, B.~Li, and Z.~Li.
\newblock Immersed-interface finite-element methods for elliptic interface
  problems with nonhomogeneous jump conditions.
\newblock {\em SIAM J. Numer. Anal.}, 46(1):472--495, 2007/08.

\bibitem{2016GuoLin}
R.~Guo and T.~Lin.
\newblock A group of immersed finite element spaces for elliptic interface
  problems.
\newblock {\em arXiv:1612.00919}, 2016.

\bibitem{2015GuzmanSanchezSarkis}
J.~Guzman, M.~A. Sanchez, and M.~Sarkis.
\newblock A finite element method for high-contrast interface problems with
  error estimates independent of contrast.
\newblock {\em arXiv:1507.03873}, 2015.

\bibitem{2009HeTHESIS}
X.~He.
\newblock {\em Bilinear immersed finite elements for interface problems}.
\newblock PhD thesis, Virginia Polytechnic Institute and State University,
  2009.

\bibitem{2008HeLinLin}
X.~He, T.~Lin, and Y.~Lin.
\newblock Approximation capability of a bilinear immersed finite element space.
\newblock {\em Numer. Methods Partial Differential Equations},
  24(5):1265--1300, 2008.

\bibitem{2009HeLinLin}
X.~He, T.~Lin, and Y.~Lin.
\newblock A bilinear immersed finite volume element method for the diffusion
  equation with discontinuous coefficient.
\newblock {\em Commun. Comput. Phys.}, 6(1):185--202, 2009.

\bibitem{2005HouLiu}
S.~Hou and X.-D. Liu.
\newblock A numerical method for solving variable coefficient elliptic equation
  with interfaces.
\newblock {\em Journal of Computational Physics}, 202(2):411--445, 2005.

\bibitem{2014JiChenLi}
H.~Ji, J.~Chen, and Z.~Li.
\newblock A symmetric and consistent immersed finite element method for
  interface problems.
\newblock {\em J. Sci. Comput.}, 61(3):533--557, 2014.

\bibitem{2010KwakWeeChang}
D.~Y. Kwak, K.~T. Wee, and K.~S. Chang.
\newblock An analysis of a broken {$P_1$}-nonconforming finite element method
  for interface problems.
\newblock {\em SIAM J. Numer. Anal.}, 48(6):2117--2134, 2010.

\bibitem{1998Li}
Z.~Li.
\newblock The immersed interface method using a finite element formulation.
\newblock {\em Appl. Numer. Math.}, 27(3):253--267, 1998.

\bibitem{2004LiLinLinRogers}
Z.~Li, T.~Lin, Y.~Lin, and R.~C. Rogers.
\newblock An immersed finite element space and its approximation capability.
\newblock {\em Numer. Methods Partial Differential Equations}, 20(3):338--367,
  2004.

\bibitem{2003LiLinWu}
Z.~Li, T.~Lin, and X.~Wu.
\newblock New {C}artesian grid methods for interface problems using the finite
  element formulation.
\newblock {\em Numer. Math.}, 96(1):61--98, 2003.

\bibitem{2001LinLinRogersRyan}
T.~Lin, Y.~Lin, R.~Rogers, and M.~L. Ryan.
\newblock A rectangular immersed finite element space for interface problems.
\newblock In {\em Scientific computing and applications ({K}ananaskis, {AB},
  2000)}, volume~7 of {\em Adv. Comput. Theory Pract.}, pages 107--114. Nova
  Sci. Publ., Huntington, NY, 2001.

\bibitem{2015LinLinZhang}
T.~Lin, Y.~Lin, and X.~Zhang.
\newblock Partially penalized immersed finite element methods for elliptic
  interface problems.
\newblock {\em SIAM J. Numer. Anal.}, 53(2):1121--1144, 2015.

\bibitem{2013LinSheenZhang}
T.~Lin, D.~Sheen, and X.~Zhang.
\newblock A locking-free immersed finite element method for planar elasticity
  interface problems.
\newblock {\em J. Comput. Phys.}, 247:228--247, 2013.

\bibitem{2015LinSheenZhang}
T.~Lin, D.~Sheen, and X.~Zhang.
\newblock {Nonconforming immersed finite element methods for elliptic interface
  problems}.
\newblock {\em arXiv:1510.00052}, 2015.

\bibitem{2012LinZhang}
T.~Lin and X.~Zhang.
\newblock Linear and bilinear immersed finite elements for planar elasticity
  interface problems.
\newblock {\em J. Comput. Appl. Math.}, 236(18):4681--4699, 2012.

\bibitem{1992RanacherTurek}
R.~Rannacher and S.~Turek.
\newblock Simple nonconforming quadrilateral {S}tokes element.
\newblock {\em Numer. Methods Partial Differential Equations}, 8(2):97--111,
  1992.

\bibitem{RenWei1994}
X.~Ren and J.~Wei.
\newblock On a two-dimensional elliptic problem with large exponent in
  nonlinearity.
\newblock {\em Transactions of the American Mathematical Society},
  343(2):749--763, 1994.

\bibitem{1982Xu}
J.~Xu.
\newblock Estimate of the convergence rate of the finite element solutions to
  elliptic equation of second order with discontinuous coefficients.
\newblock {\em Natural Science Journal of Xiangtan University}, 1:1--5, 1982.

\bibitem{2013ZhangTHESIS}
X.~Zhang.
\newblock {\em Nonconforming {I}mmersed {F}inite {E}lement {M}ethods for
  {I}nterface {P}roblems}.
\newblock PhD thesis, Virginia Polytechnic Institute and State University,
  2013.

\end{thebibliography}

\end{document}